\documentclass[11pt]{amsart}

\usepackage{amssymb,upref}
\usepackage[mathcal]{euscript}
\usepackage[cmtip,all]{xy}

\newcommand{\calJ}{\mathcal{J}}

\newcommand{\alg}[1]{\mathrm{alg}\left\{#1\right\}}
\newcommand{\espan}[1]{\mathrm{span}\left\{#1\right\}}

\newcommand{\FF}{\mathbb{F}}

\newtheorem{theorem}{Theorem}[section]
\newtheorem{proposition}[theorem]{Proposition}

\newtheorem{corollary}[theorem]{Corollary}

\theoremstyle{definition}

\newtheorem{example}[theorem]{Example}

\theoremstyle{remark}
\newtheorem{remark}[theorem]{Remark}

\begin{document}

\title{On some Jordan baric algebras}

\author[Alberto Elduque]{Alberto Elduque$^{\star}$}
\thanks{$^{\star}$ Supported by the Spanish Ministerio de Educaci\'{o}n y Ciencia and
FEDER (MTM 2010-18370-C04-02) and by the Diputaci\'on General de Arag\'on (Grupo de Investigaci\'on de \'Algebra)}
\address{Departamento de Matem\'aticas e Instituto Universitario de Matem\'aticas y Aplicaciones,
Universidad de Zaragoza, 50009 Zaragoza, Spain}
\email{elduque@unizar.es}

\author[Alicia Labra]{Alicia Labra$^{\star\star}$}
\thanks{$^{\star\star}$Supported by Universidad de Chile, Proyecto ENL 11 / 09, Enlace con Concurso Fondecyt Regular VID 2011.}
\address{Departamento de Matem\'aticas,
Facultad de Ciencias, Universidad de Chile.  Casilla 653, Santiago, Chile}
\email{alimat@uchile.cl}




\begin{abstract}
Several classes of baric algebras studied by different authors will be given a unified treatment, using the technique of gametization introduced by Mallol
et al. \cite{MVB}. Many of these algebras will be shown to be either Jordan algebras or to be closely related to them.
\end{abstract}

\maketitle


\section{Introduction}

In \cite{EO} the class of commutative algebras satisfying the condition (`adjoint identity') $x^2x^2=N(x)x$ for any $x$, where $N$ denotes a nonzero cubic form, were shown to be closely related to certain Jordan algebras of generic degree at most $4$. Among these algebras, it was noticed in the last remarks of \cite{EO} that those in which the cubic norm $N$ takes the form $N(x)=\omega(x)^3$, for a nonzero algebra homomorphism $\omega$ into the ground field, may have a signification in genetic. These algebras had been considered in \cite{W}, where $\ker\omega$ was shown to be a Jordan algebra. However, Mallol and Varro showed that, in general, these algebras are not power-associative \cite{MV}.

In the following, let $\FF$ be a ground field of characteristic $\ne 2$ containing at least $4$ elements (we are just excluding the field of three elements in characteristic $3$). This will insure that the conditions satisfied by the algebras under consideration remain valid when extending scalars.
A \emph{baric} algebra over $\FF$ is a nonassociative but commutative algebra $A$ over $\FF$ endowed with a nonzero algebra homomorphism $\omega:A\rightarrow \FF$. Baric algebras were introduced by Etherington \cite{Eth1,Eth2}, as an algebraic framework for certain problems in genetic.

In \cite{AL} the class of baric algebras satisfying $x^2x^2=\omega(x)x^3$ for any $x$ was considered. It was shown that these algebras are always Jordan algebras, with arguments based on the study of Peirce decompositions relative to idempotents. Also, in \cite{LR2} another class of baric algebras, those satisfying $x^2x^2=3\omega(x)x^3-3\omega(x)^2x^2+\omega(x)^3x$ for any $x$, were shown to be Jordan algebras too.

In this paper, all these baric algebras will be given a unified treatment. More precisely, the baric algebras $(A,\omega)$ satisfying one of the following conditions will be considered:
\begin{align}
x^2x^2&=\omega(x)^3x,\label{eq:omega3}\\
x^2x^2&=\omega(x)x^3,\label{eq:omega}\\
x^2x^2&=3\omega(x)x^3-3\omega(x)^2x^2+\omega(x)^3x,\label{eq:omegas}\\
x^2x^2&=2\omega(x)x^3-\omega(x)^2x^2,\label{eq:omega2s}
\end{align}
for any $x\in A$.

It was shown in \cite{BCOZ} how to move from algebras satisfying \eqref{eq:omega} to algebras satisfying \eqref{eq:omegas} and back. Actually, this is a particular case of the \emph{gametization process} considered in \cite{MVB}, which will play a key role here. Next section will be devoted to this process.

As mentioned above, algebras satisfying \eqref{eq:omega3} were considered in \cite{W} and in \cite{EO}. In \cite{MV} these algebras were shown not to be Jordan algebras in general, idempotents were shown to exist and properties of the Peirce decompositions relative to idempotents of these algebras were studied.

Peirce decompositions of the baric algebras satisfying \eqref{eq:omegas} were considered in \cite{LR2}. Later on, it was proved in \cite{MS} that if $A$ is a baric algebra satisfying an equation of the form
$x^2 x^2 = \alpha w(x)x^3 + \beta w(x)^2 x^2 + \gamma w(x)^3 x$, for scalars $\alpha,\beta,\gamma\in\FF$ (which must satisfy $\alpha+\beta+\gamma=1$), then $A$ is power-associative if and only if it is a Jordan algebra, if and only if $A$ satisfies either \eqref{eq:omega} or \eqref{eq:omegas}.

As for baric algebras satisfying \eqref{eq:omega2s}, the following example in \cite{BCOZ} shows that they are not Jordan algebras in general.

\begin{example}
Let $A$ be the commutative algebra
with a basis $ \{e,u_1,u_2,u_3,$ $ u_4,s,t\}$ and multiplication determined by
\begin{gather*}
e^2 = e, \quad eu_i = \frac{1}{2}u_i \; \forall \; i=1,2,3,4,  \quad et = t, \\ u_1u_4 = -s-t, \quad u_1s = u_1t = u_3, \\  u_2u_3 = s + t, \quad u_2s = u_2t = u_4,\quad \text{other products being zero.}
\end{gather*}
and $\omega: A \rightarrow \FF$ given by $\omega(e) =1$, $\omega(u_i) = \omega(s) = \omega(t) = 0$.

Then $(A,\omega)$ is a baric algebra  satisfying \eqref{eq:omega2s}. But  it is not a Jordan algebra, since for $x = u_1 + s$ and $y = u_2 $, we have $(x^2y)x = 4u_3 \neq x^2(yx) = 0$.
\end{example}

\bigskip
We will use the technique of gametization  introduced in \cite{MVB} to find relations among the baric algebras considered here. The gametization process will be reviewed in Section \ref{se:gametization}. Then, in Section \ref{se:Jordan}, the way to obtain a Jordan algebra of generic degree at most $4$ in \cite{EO} will be reviewed, and this will be used to prove in a very simple way that baric algebras satisfying \eqref{eq:omega} or \eqref{eq:omegas} are always Jordan algebras, while algebras satisfying \eqref{eq:omega3} are very close to being such. A characterization of the nonzero idempotent elements in the algebras satisfying \eqref{eq:omega3}, \eqref{eq:omega} or \eqref{eq:omegas} will be given in Section \ref{se:idempotents}. Finally, Section \ref{se:omega2s} will be devoted to the baric algebras satisfying equation \eqref{eq:omega2s}.


\section{Gametization}\label{se:gametization}

In this section the process of \emph{gametization} in \cite{MVB} will be reviewed in a way suitable for our purposes. Given a baric algebra $(A,\omega)$ over a field $\FF$, with multiplication denoted by juxtaposition, and a scalar $1\ne\gamma\in\FF$, the \emph{$\gamma$-gametization} of $A$ is the algebra $A_\gamma$, where $A_\gamma$ coincides with $A$ as a vector space, but with the new multiplication given by
\[
x\bullet y=(1-\gamma)xy+\frac{1}{2}\gamma\bigl(\omega(x)y+\omega(y)x\bigr),
\]
for any $x,y\in A$. It is clear that $\omega$ is a homomorphism too from $A_\gamma$ into $\FF$, so $(A_\gamma,\omega)$ is again a baric algebra.

\begin{proposition} Let $(A,\omega)$ be a baric algebra over the field $\FF$. Then
$(A_\gamma)_\delta=A_{\gamma+\delta-\gamma\delta}$ for any $1\ne \gamma,\delta\in\FF$. In particular $(A_\gamma)_{\frac{\gamma}{\gamma -1}}=A$, so $(A_2)_2=A$.
\end{proposition}
\begin{proof}
Denote by $\bullet$ the multiplication in $A_\gamma$ and by $\diamond$ the multiplication in $(A_\gamma)_\delta$. Then, for any  $x\in A$:
\[
\begin{split}
x\diamond x&=(1-\delta)x\bullet x+\delta\omega(x)x\\
 &=(1-\delta)(1-\gamma)x^2+((1-\delta)\gamma+\delta)\omega(x)x\\
 &=(1-(\gamma+\delta-\gamma\delta))x^2+(\gamma+\delta-\gamma\delta)\omega(x)x,
\end{split}
\]
whence the result.
\end{proof}

Given an algebra $A$, the \emph{unitization} of $A$ is the algebra $A^\sharp$ obtained by adding a formal unity to $A$: $A^\sharp=\FF 1\oplus A$, with multiplication obtained by extending the multiplication on $A$ by the condition $1x=x1=x$ for any $x\in A^\sharp$. The algebra $A$ becomes then an ideal of its unitization $A^\sharp$.

Algebras obtained by $2$-gametization are close to the original ones:

\begin{proposition}\label{pr:unitizationsAA2}
Given a baric algebra $(A,\omega)$, the unitizations  of $A$ and its $2$-gametization $A_2$ are isomorphic.
\end{proposition}
\begin{proof}
Consider the linear map $\varphi: A^\sharp\rightarrow (A_2)^\sharp$, defined by $\varphi(1)=1$ and $\varphi(x)=\omega(x)1-x$. Then, for any $x\in A$, we have
\[
\begin{split}
\varphi(x)^{\bullet 2}&=(\omega(x)1-x)^{\bullet 2}
=\omega(x)^21+\bigl(-2\omega(x)x+x^{\bullet 2}\bigr)\\
&=\omega(x^2)1-x^2
=\varphi(x^2).
\end{split}
\]
Since $A$ is commutative, it follows that $\varphi$ is a homomorphism, which is clearly bijective.
\end{proof}

\begin{corollary}
Let $(A,\omega)$ be a baric algebra. Then $A$ is a Jordan algebra if and only if so is its $2$-gametization $A_2$.
\end{corollary}
\begin{proof}
An algebra $A$ is a Jordan algebra if and only if so is its unitization. Now the result follows from the preceding proposition.
\end{proof}

The next result relates the conditions \eqref{eq:omega3}, \eqref{eq:omega} and \eqref{eq:omegas} and the gametization process.

\begin{proposition}\label{pr:gametization}
Let $(A,\omega)$ be a baric algebra. Then the following conditions hold:
\begin{itemize}
\item $A$ satisfies \eqref{eq:omega} if and only if $A_{-1}$ satisfies \eqref{eq:omega3}.
\item $A$ satisfies \eqref{eq:omega} if and only if $A_2$ satisfies \eqref{eq:omegas}.
\end{itemize}
\end{proposition}
\begin{proof}
Denote by juxtaposition the multiplication in $A$, and by $\bullet$ and $\diamond$ the multiplications in the gametizations $A_{-1}$ and $A_2$ respectively. Then for any $x\in A$:
\[
\begin{split}
x^{\bullet 2}\bullet x^{\bullet 2}
 &=2(x^{\bullet 2})^2-\omega(x)^2x^{\bullet 2}\\
 &=2\bigl(2x^2-\omega(x)x\bigr)^2-\omega(x)^2\bigl(2x^2-\omega(x)x\bigr)\\
 &=8x^2x^2-8\omega(x)x^3+2\omega(x)^2x^2-2\omega(x)^2x+\omega(x)^3x\\
 &=8\bigl(x^2x^2-\omega(x)x^3\bigr)+\omega(x)^3x,
\end{split}
\]
so that $A$ satisfies \eqref{eq:omega} if and only if $A_{-1}$ satisfies \eqref{eq:omega3}.

In the same vein we compute:
\[
\begin{split}
x^{\diamond 2}\diamond x^{\diamond 2}
 &=-(-x^2+2\omega(x)x)^2+2\omega(x)^2(-x^2+2\omega(x)x)\\
 &=-x^2x^2+4\omega(x)x^3-6\omega(x)x^2x^2+4\omega(x)^3x,
\end{split}
\]
and hence, since $x^{\diamond 3}=x^{\diamond 2}\diamond x=-x^{\diamond 2}x+\omega(x)x^{\diamond 2}+\omega(x)^2x=x^3-3\omega(x)x^2+3\omega(x)^2x$, we obtain
\[
x^{\diamond 2}\diamond x^{\diamond 2}-\omega(x)x^{\diamond 3}=
-\bigl(x^2x^2-3\omega(x)x^3+3\omega(x)^2x^2-\omega(x)^3x\bigr),
\]
so that $A$ satisfies \eqref{eq:omegas} if and only if $A_2$ satisfies \eqref{eq:omega}. As $(A_2)_2=A$, this finishes the proof.
\end{proof}


\section{Jordan baric algebras}\label{se:Jordan}

Recall from \cite{EO} that if $A$ is a commutative algebra over an infinite field of characteristic different from $2$, endowed with a symmetric bilinear form $\langle .\mid .\rangle$ which is `associative' (i.e., $\langle xy\mid z\rangle=\langle x\mid yz\rangle$ for any $x,y,z$), satisfies $x^2x^2=\langle x\mid x^2\rangle x$ for any $x$, and such that the cubic form $\langle x\mid x^2\rangle$ is nonzero, then by linearization we obtain:
\begin{equation}\label{eq:4x2xy}
4x^2(xy)=3\langle x^2\mid y\rangle x+\langle x^2\mid x\rangle y,
\end{equation}
for any $x,y\in A$. Substitute $x$ by $x^2$ in \eqref{eq:4x2xy} to get:
\[
4\langle x\mid x^2\rangle x(x^2y)=3\langle x\mid x^2\rangle\langle x\mid y\rangle x^2 +\langle x\mid x^2\rangle \langle x\mid x^2\rangle y.
\]
The set of those elements $x\in A$ with $\langle x\mid x^2\rangle \ne 0$ is dense in the Zariski topology, so we conclude:
\begin{equation}\label{eq:4xx2y}
4x(x^2y)=3\langle x\mid y\rangle x^2+\langle x\mid x^2\rangle y,
\end{equation}
for any $x,y\in A$. (For the definition and main features of the Zariski topology on not necessarily finite dimensional spaces, the reader may consult \cite{McC}.)

On the formal direct sum $\calJ(A)=\FF 1\oplus A$ define (\cite[Equation (20)]{EO}) the commutative multiplication given by:
\begin{equation}\label{eq:J(A)}
1\cdot 1=1,\ 1\cdot x=x,\ x\cdot y=\frac{3}{4}\langle x\mid y\rangle 1+ xy,
\end{equation}
for any $x,y\in A$. Then it is shown in \cite{EO} that $\calJ(A)$ is a Jordan algebra. Let us include the proof for completeness. Take $X=\alpha 1+x$, $Y=\beta 1+y$ in $\calJ(A)$. Then, if $(X,Y,Z)^\cdot=(X\cdot Y)\cdot Z-X\cdot (Y\cdot Z)$ denotes the associator on $\calJ(A)$, by commutativity $(X,Y,X)^\cdot =0=(1,X,Y)^\cdot=(X,1,Y)^\cdot$, and hence
\[
\begin{split}
(X^{\cdot 2},Y,X)^\cdot&=(x^2,y,x)^\cdot\\
 &=(x^2\cdot y)\cdot x-x^2\cdot (y\cdot x)\\
 &=(x^2y)\cdot x+\frac{3}{4}\langle x^2\mid y\rangle x-x^2\cdot (yx)-\frac{3}{4}\langle x\mid y\rangle x^2\\
 &=(x^2y)x+\frac{3}{4}\langle x^2y\mid x\rangle 1+\frac{3}{4}\langle x^2\mid y\rangle x
 -x^2(yx)-\frac{3}{4}\langle x^2\mid yx\rangle 1 \\
& \quad-\langle x\mid y\rangle x^2\\
 &=(x^2y)x-x^2(yx)+\frac{3}{4}\langle x^2\mid y\rangle x-\frac{3}{4}\langle x\mid y\rangle x^2,
\end{split}
\]
and this is $0$ by equations \eqref{eq:4x2xy} and \eqref{eq:4xx2y}.

As an immediate consequence we have the following result:

\begin{proposition}\label{pr:omega3Jordan}
Let $(A,\omega)$ be a baric algebra satisfying \eqref{eq:omega3}, then the vector space $\calJ(A)=\FF 1\oplus A$, with multiplication given by
\[
1\cdot 1=1,\ 1\cdot x=x,\ x\cdot y=\frac{3}{4}\omega(x)\omega(y) 1+ xy,
\]
is a Jordan algebra.
\end{proposition}
\begin{proof}
We may extend scalars to insure that the ground field is infinite, and then the hypotheses above are satisfied with $\langle x\mid y\rangle =\omega(x)\omega(y)$.
\end{proof}

Therefore, the baric algebras satisfying \eqref{eq:omega3} are close to be Jordan algebras. For the baric algebras satisfying \eqref{eq:omega} or \eqref{eq:omegas} the situation is even better:

\begin{theorem}\label{th:omegasJordan}
Let $(A,\omega)$ be a baric algebra. Then the following conditions hold:
\begin{itemize}
\item If $(A,\omega)$ satisfies \eqref{eq:omega}, then its unitization $A^\sharp$ is isomorphic to $\calJ(A_{-1})$. (Recall that in this case $A_{-1}$ satisfies \eqref{eq:omega3} by Proposition \ref{pr:gametization}.)
\item If $(A,\omega)$ satisfies \eqref{eq:omegas}, then its unitization $A^\sharp$ is isomorphic to $\calJ(A_3)$. (Here the gametization $A_3=(A_2)_{-1}$ also satisfies \eqref{eq:omega3} by Proposition \ref{pr:gametization}.)
\item If $(A,\omega)$ satisfies either \eqref{eq:omega} or \eqref{eq:omegas}, then $A$ is a Jordan algebra.
\end{itemize}
\end{theorem}
\begin{proof}
If $(A,\omega)$ satisfies \eqref{eq:omega}, denote by $\bullet$ the product in its gametization $A_{-1}$, so that $x\bullet x=2x^2-\omega(x)x$, or $x^2=\frac{1}{2}\bigl(x\bullet x+\omega(x)x\bigr)$. Then the linear map $\Phi:A\rightarrow \calJ(A_{-1})$, $x\mapsto \frac{1}{4}\omega(x)1+\frac{1}{2}x$, is an algebra homomorphism, since
\[
\begin{split}
\Phi(x)^2&=\Bigl(\frac{1}{4}\omega(x)1+\frac{1}{2}x\Bigr)^2\\
 &=\left(\frac{1}{16}\omega(x)^2+\frac{1}{4}\frac{3}{4}\omega(x)^2\right)1+\left(\frac{1}{4}x\bullet x +2\frac{1}{4}\frac{1}{2}\omega(x)x\right)\\
 &=\frac{1}{4}\omega(x)^21+\frac{1}{4}\left(x\bullet x+\omega(x)x\right)\\
 &=\frac{1}{4}\omega(x^2)1+\frac{1}{2}x^2=\Phi(x^2).
\end{split}
\]
This map $\Phi$ extends trivially to an algebra isomorphism $\Phi:A^\sharp\rightarrow \calJ(A_{-1})$ by taking $\Phi(1)=1$, thus proving the first part.

For the second part note that if $(A,\omega)$ satisfies \eqref{eq:omegas}, then its unitization $A^\sharp$ is isomorphic to the unitization $(A_2)^\sharp$ (Proposition
\ref{pr:unitizationsAA2}), and we apply the first part.

Finally, the third part is a direct consequence of the previous two parts and Proposition \ref{pr:omega3Jordan}.
\end{proof}

As mentioned in the Introduction, this theorem provides a unified and simple proof of the fact that the baric algebras satisfying either \eqref{eq:omega} or \eqref{eq:omegas} are indeed Jordan algebras, with no need to rely on properties of Peirce decompositions.


\section{Idempotents}\label{se:idempotents}

Let $(A,\omega)$ be a baric algebra satisfying \eqref{eq:omega}. If $e$ is a nonzero idempotent of $A$, then $e=e^2e^2=\omega(e)e^3=\omega(e)e$, so that $\omega(e)=1$ and $e=e^3$. On the other hand,  since $A$ is a  Jordan algebra (Theorem \ref{th:omegasJordan}), it is power-associative, so for any $x\in A$ we have $x^3x^3=x^6=(x^2x^2)x^2=\omega(x)x^3x^2=\omega(x)(x^2x^2)x=\omega(x)^2x^4=\omega(x)^3x^3$. In particular, if $\omega(x)=1$ we obtain that $x^3$ is an idempotent. This proves the first part of the following result:

\begin{theorem}\label{th:idempotents}
Let $(A,\omega)$ be a baric algebra.
\begin{itemize}
\item
If $(A,\omega)$ satisfies \eqref{eq:omega}, then the set of nonzero idempotents of $A$ is the set $\{x^3:\omega(x)=1\}$.
\item
If $(A,\omega)$ satisfies \eqref{eq:omega3}, then the set of nonzero idempotents of $A$ is the set $\{\frac{1}{4}x^3+\frac{3}{8}x^2+\frac{3}{8}x : \omega(x)=1\}$.
\item
If $(A,\omega)$ satisfies \eqref{eq:omegas}, then the set of nonzero idempotents of $A$ is the set $\{x^3-3x^2+3x:\omega(x)=1\}$.
\end{itemize}
\end{theorem}
\begin{proof}
Note first that if $(A,\omega)$ is a baric algebra satisfying one of the conditions \eqref{eq:omega}, \eqref{eq:omega3} or \eqref{eq:omegas}, then any nonzero idempotent $e$ satisfies $\omega(e)=1$. Also, if $e$ is an idempotent of a baric algebra $(A,\omega)$ with $\omega(e)=1$, then $e$ remains an idempotent in any gametization $A_\gamma$.

Now, the first part has already been proved. For the second part, if $(A,\omega)$ satisfies \eqref{eq:omega3} and $e$ is an idempotent in $A$, then $e=x^{\bullet 3}$ for some $x\in A$ with $\omega(x)=1$, where $x\bullet y$ denotes the multiplication in the gametization $A_{\frac{1}{2}}$. (Note that $\frac{1}{2}=\frac{-1}{-1-1}$, so $(A_{\frac{1}{2}})_{-1}=A$.) But we have $x^{\bullet 2}=\frac{1}{2}x^2+\frac{1}{2}\omega(x)x$, and $x^{\bullet 3}=\frac{1}{2}x^{\bullet 2}x+\frac{1}{4}(\omega(x^2)x+\omega(x)x^{\bullet 2})
=\frac{1}{4}x^3+\frac{3}{8}\omega(x)x^2+\frac{3}{8}\omega(x)^2x$. This proves the second part, while the proof of the last part is similar with $A_2$ instead of $A_{\frac{1}{2}}$.
\end{proof}

In particular, this theorem proves the existence of idempotents for all these baric algebras.

Given a nonzero idempotent $e$ in a baric algebra satisfying \eqref{eq:omega}, we may use the well-known properties of the Peirce decomposition for idempotents in Jordan algebras, but the situation here is very simple. The equation $x^2x^2=\omega(x)x^3$ gives, by linearization
\begin{align}
&4x^2(xy)=\omega(y)x^3+\omega(x)\bigl(2x(xy)+x^2y\bigr),\label{eq:linearization1omega}\\[4pt]
&2x^2y^2+4(xy)^2=\omega(x)\bigl(2y(yx)+y^2x\bigr)
 +\omega(y)\bigl(2x(xy)+x^2y\bigr),\label{eq:linearization2omega}\\[4pt]
&4x^2(yz)+8(xy)(xz)=2\omega(x)\bigl(x(zy)+z(xy)+y(xz)\bigr) \nonumber\\
&\qquad\qquad\qquad\qquad +\omega(y)\bigl(2x(xz)+x^2z\bigr)+\omega(z)\bigl(2x(xy)+x^2y\bigr),
\label{eq:linearization3omega}
\end{align}
for any $x,y,z\in A$.

\begin{proposition}\label{le:Peirce_omega}
Let $(A,\omega)$ be a baric algebra satisfying \eqref{eq:omega}, and let $e$ be a nonzero idempotent of $A$. Denote $\ker\omega$ by $N$. Then $N=N(\frac{1}{2})\oplus N(0)$, where $N(\lambda)=\{x\in N: ex=\lambda x\}$. Moreover, the following conditions hold:
\begin{itemize}
\item $N(\frac{1}{2})^2+N(0)^2\subseteq N(0)$, $N(\frac{1}{2})N(0)\subseteq N(\frac{1}{2})$,
\item $x^3=0$ for any $x\in N(\frac{1}{2})\cup N(0)$, $x^2(xy)=0$ for any $x\in N$,
\item
$u^2v^2=-2(uv)^2$, $u^2v=2u(uv)$, $uv^2=2(uv)v$ for any $u\in N(\frac{1}{2})$ and $v\in N(0)$.
\item The set of nonzero idempotents is $\{e+u+u^2: u\in N(\frac{1}{2})\}$.
\end{itemize}
\end{proposition}
\begin{proof}
Equation \eqref{eq:linearization1omega} with $x=e$ and $y\in N$ gives $2e(ey)-ey=0$ for any $y\in N$, and hence $N=N(\frac{1}{2})\oplus N(0)$ as required, while with $y=e$ gives $x^3=0$ for any $x\in N(\frac{1}{2})\cup N(0)$, and with $x=y\in N$ gives $x^2(xy)=0$ for any  $x,y\in N$.

Now equation \eqref{eq:linearization2omega} with $x=e$ and $y\in N(\frac{1}{2})\cup N(0)$ gives $N(\frac{1}{2})^2+N(0)^2\subseteq N(0)$, while \eqref{eq:linearization3omega} with $x=e$, $y=u\in N(\frac{1}{2})$ and $z=v\in N(0)$ shows $N(\frac{1}{2})N(0)\subseteq N(\frac{1}{2})$. Equation \eqref{eq:linearization2omega} with $x=u\in N(\frac{1}{2})$ and $y\in N(0)$ gives $u^2v^2=-2(uv)^2$ for any $u\in N(\frac{1}{2})$ and $v\in N(0)$.

Equation \eqref{eq:linearization3omega} with $z=e$, $x=u\in N(\frac{1}{2})$ and $y=v\in N(0)$ gives $u^2v=2u(uv)$,  with $z=e$ $x=v\in N(0)$ and $y=u\in N(\frac{1}{2})$ gives $uv^2=2(uv)v$, and with $x=e$, $y=u\in N(\frac{1}{2})$ and $z=v\in N(0)$ gives $N(\frac{1}{2})N(0)\subseteq N(\frac{1}{2})$.

Finally, if $x$ is a nonzero idempotent, then $\omega(x)=1$ so $x=e+u+v$ for some $u\in N(\frac{1}{2})$ and $v\in N(0)$. Then $e+u+v=(e+u+v)^2=e+(u+2uv)+(u^2+v^2)$, so $u=u+2uv$ and $uv=0$, and $u^2+v^2=v$. Multiply this last equation by $v$ to get $v^2=u^2v=2u(uv)=0$, so $v=u^2$. Conversely, given any $u\in N(\frac{1}{2})$, $u^3=0=eu^2$, so $e+u+u^2$ is a nonzero idempotent.
\end{proof}

Now the properties of the Peirce decompositions of those baric algebras satisfying either \eqref{eq:omega3} or \eqref{eq:omegas} (which have been studied in \cite{W} and \cite{LR2}) follow immediately by gametization:

\begin{corollary}
Let $(A,\omega)$ be a baric algebra, and let $e$ be a nonzero idempotent of $A$. Denote $\ker \omega$ by $N$, and for any $\lambda\in \FF$ write $N_\lambda=\{x\in N:ex=\lambda x\}$. Then:
\begin{enumerate}
\item If $(A,\omega)$ satisfies \eqref{eq:omega3}, then
$N=N(\frac{1}{2})\oplus N(-\frac{1}{2})$. Moreover, the following conditions hold:
\begin{itemize}
\item $N(\frac{1}{2})^2+N(-\frac{1}{2})^2\subseteq N(-\frac{1}{2})$, $N(\frac{1}{2})N(-\frac{1}{2})\subseteq N(\frac{1}{2})$,
\item
$x^3=0$ for any $x\in N(\frac{1}{2})\cup N(-\frac{1}{2})$, $x^2(xy)=0$ for any $x\in N$,
\item
$u^2v^2=-2(uv)^2$, $u^2v=2u(uv)$, $uv^2=2(uv)v$ for any $u\in N(\frac{1}{2})$ and $v\in N(-\frac{1}{2})$.
\item The set of nonzero idempotents is $\{e+u+\frac{1}{2}u^2: u\in N(\frac{1}{2})\}$.
\end{itemize}
\item If $(A,\omega)$ satisfies \eqref{eq:omegas}, then $N=N(\frac{1}{2})\oplus N(1)$. Moreover, the following conditions hold:
\begin{itemize}
\item $N(\frac{1}{2})^2+N(1)^2\subseteq N(1)$, $N(\frac{1}{2})N(1)\subseteq N(\frac{1}{2})$,
\item
$x^3=0$ for any $x\in N(\frac{1}{2})\cup N(1)$, $x^2(xy)=0$ for any $x\in N$,
\item
$u^2v^2=-2(uv)^2$, $u^2v=2u(uv)$, $uv^2=2(uv)v$ for any $u\in N(\frac{1}{2})$ and $v\in N(1)$.
\item The set of nonzero idempotents is $\{e+u-u^2: u\in N(\frac{1}{2})\}$.
\end{itemize}
\end{enumerate}
\end{corollary}
\begin{proof}
For $(A,\omega)$ satisfying \eqref{eq:omega3}, the gametization $(A_{\frac{1}{2}},\omega)$, with multiplication given by $x\bullet y=\frac{1}{2}xy+\frac{1}{4}(\omega(x)y+\omega(y)x)$, satisfies \eqref{eq:omega}. Also, for $x\in N$, $e\bullet x=\frac{1}{2}ex+\frac{1}{4}x$. Hence $N(\frac{1}{2})^\bullet:=\{x\in N:e\bullet x=\frac{1}{2}x\}=\{x\in N: ex=\frac{1}{2}x\}=N(\frac{1}{2})$, while $N(0)^\bullet:=\{x\in N:e\bullet x=0\}=\{x\in N:ex=-\frac{1}{2}x\}=N(-\frac{1}{2})$. Moreover, for $x,y\in N$, $x\bullet y=\frac{1}{2}xy$. Then the conditions in the first item are immediate consequences of Proposition \ref{le:Peirce_omega}.

The proof of the second item is similar, but now one has to consider the gametization $(A_2,\omega)$.
\end{proof}


\section{Baric algebras satisfying \eqref{eq:omega2s}}\label{se:omega2s}

Identity \eqref{eq:omega2s} appeared for the first time in \cite{Eth2}, where it was excluded from the analysis, as well as in \cite{MS}, \cite{LZ} and \cite{BH}.

These algebras are characterized in the next result, which uses a variant of the gametization process.

\begin{theorem}
Let $(A,\omega)$ be a baric algebra satisfying \eqref{eq:omega2s}. Define a new commutative multiplication on $A$ by $x\ast y=xy-\frac{1}{2}(\omega(x)y+\omega(y)x)$. Then $A^{\ast 2}\ne A$ and $x^{\ast 2}\ast x^{\ast 2}=0$ for any $x\in A$.

Conversely, let $A$, with multiplication $x\ast y$, be a commutative algebra satisfying $A^{\ast 2}\ne A$ and $x^{\ast 2}\ast x^{\ast 2}=0$ for any $x\in A$. Let $\omega$ be any nonzero linear form with $A^{\ast 2}\subseteq \ker\omega$. Then $A$, with multiplication given by $xy=x\ast y-\frac{1}{2}(\omega(x)y+\omega(y)x)$ is a baric algebra satisfying \eqref{eq:omega2s}.
\end{theorem}
\begin{proof}
If $(A,\omega)$ is a baric algebra satisfying \eqref{eq:omega2s} and $x\ast y=xy-\frac{1}{2}(\omega(x)y+\omega(y)x)$ for any $x,y\in A$, then $\omega(x\ast y)= 0$ for any $x,y\in A$, so $A^{\ast 2}\ne A$. Besides, since $x^{\ast 2}\in\ker\omega$, we have
\[
x^{\ast 2}\ast x^{\ast 2}=(x^{\ast 2})^2=(x^2-\omega(x)x)^2=0,
\]
because of \eqref{eq:omega2s}.

The arguments can be reversed to give the converse part.
\end{proof}

Given a baric algebra $(A,\omega)$ satisfying \eqref{eq:omega2s}, and the multiplication $x\ast y$ in the previous theorem, a nonzero element $e\in A$ is an idempotent if and only if $e^2=e^{\ast 2}+\omega(e)e=e$. But $\omega(e)=\omega(e)^2$, so either $\omega(e)=1$ and hence $e^{\ast 2}=0$, or $\omega(e)=0$ and $e=e^{\ast 2}=(e^{\ast 2})^{\ast 2}=0$, a contradiction. Hence the nonzero idempotent elements in $A$ are those elements $e\in A$ with $\omega(e)=1$ and $e^{\ast 2}=0$.

\begin{example}
Let $V$ be a vector space and let $S^2(V)$ be its second symmetric power. On $A=V\oplus S^2(V)$ define a commutative multiplications $\ast$ by letting $x\ast x$ be the class of $x\otimes x$ in $S^2(V)=V\otimes V/\textrm{ideal}\langle u\otimes v-v\otimes u: u,v\in V\rangle$, and $A\ast S^2(V)=0$. Then $(A^{\ast 2})^{\ast 2}=0$. Hence, if $\omega$ is any nonzero linear form with $S^2(V)\subseteq \ker\omega$, $(A,\omega)$ is a baric algebra with multiplication $xy=x\ast y-\frac{1}{2}(\omega(x)y+\omega(y)x)$, which satisfies \eqref{eq:omega2s}. If $e$ is any element satisfying $\omega(e)=1$, then $e$ does not belong to $S^2(V)$, and hence $e^{\ast 2}\ne 0$. Therefore, the baric algebra $(A,\omega)$ contains no nonzero idempotents.
\end{example}

We finish the paper posing the following natural conjecture:

\noindent\textbf{Conjecture:}\quad Any finite dimensional commutative algebra satisfying the identity $(x^2)^2=0$ is solvable.

This conjecture has been checked for low dimensional algebras in \cite{BG}.

\begin{example}
Let $A$ be the two-dimensional algebra with a basis $\{a,b\}$ and commutative multiplication given by $a^2=ab=b$, $b^2=0$. Then $(A^2)^2=0$ (so $A$ is solvable), but $A$ is not a nil algebra, since $a^n=b$ for any $n$, where $a^{n+1}=a^na$ for any $n$.
\end{example}

In \cite[pp. 82, 84 and 127]{ZSSS} there are examples of Jordan and alternative algebras $A$ satisfying $(A^2)^2=0$ which are not nilpotent.

\begin{example}
Let $A$ be the commutative algebra with a basis $\{a,b,c,d,e,f\}$ and multiplication given by $a^2=d$, $b^2=e$, $ab=c$, $c^2=f$, $de=-2f$ and all the other products equal to $0$. Then $A^2=\espan{c,d,e,f}$, $(A^2)^2=\espan{f}\ne 0$ and $((A^2)^2)^2=0$. Moreover, for $x=\alpha a+\beta b+\gamma c+\delta d+\epsilon e+\varphi f$, $x^2=\alpha^2 d+\beta^2 e+2\alpha\beta c-4\delta\epsilon f$, and $(x^2)^2=2\alpha^2\beta de+(2\alpha\beta)^2c^2=(-4\alpha^2\beta^2+4\alpha^2\beta^2)f=0$. Thus $A$ satisfies $(x^2)^2=0$ for any $x$, but $(A^2)^2\ne 0$.
\end{example}

\begin{remark}
If $A$ is a commutative algebra over a field of characteristic $\ne 2,3$ satisfying $(x^2)^2=0$ for any $x\in A$, then the subalgebra $\alg{x}$ generated by any $x$ satisfies $(\alg{x}^2)^2=0$.
\end{remark}
\begin{proof}
Define $x^{n+1}=x^nx$ for any $n\geq 1$. By linearization we have $(x^2)(xy)=0$ for any $x,y\in A$, and hence $x^2(yz)=-2(xz)(xy)$ for any $x,y,z\in A$. Let us prove that for $n,m\geq 2$, $x^nx^m=0$ for any $x\in A$. If $n=2$, $x^nx^m=x^2(xx^{m-1})\in x^2(xA)=0$. In the same vein, if $m=2$, $x^nx^m=0$. Finally, if $n,m>2$, $x^nx^m=(xx^{n-1})(xx^{m-1})=-\frac{1}{2} x^2(x^{n-1}x^{m-1})$, and an inductive argument gives $x^nx^m=0$.

We conclude that $\alg{x}^2=\espan{x^n: n\geq 2}$, and $(\alg{x}^2)^2=0$.
\end{proof}



\begin{thebibliography}{99}

\bibitem{AL} R.~Andrade, A.~Labra, {\it On a class of baric algebras}  Linear algebra  Appl. \textbf{245}, (1996), 49--53.

\bibitem{BCOZ} J.~Bayara, A.~Conseibo, M.~Ouattara, F.~Zitan {\it Power-associative algebars that are train algebras}
J.~Algebra \textbf{324} (2010), 1159--1176.

\bibitem{BG} A.~Behn and H.~Guzzo Jr., {\it On solvability of the commutative algebras which satisfy $(x^2)^2=0$}, preprint.

\bibitem{BH} A.~Behn, I.R.~Hentzel, {\it Idempotents in plenary train
algebras,} J.~Algebra \textbf{324} (2010), 3241-3248.

\bibitem{EO} A.~Elduque, S.~Okubo, {\it On algebras satisfying $x^2x^2=N(x)x,$} Mathematische Zeitschrift \textbf{235}, (2000) 2, 275-314.

\bibitem{Eth1} I.M.~Etheringhton, {\it Genetic Algebras} Proc.~Roy.~Soc.~ Edinb. \textbf{59}, (1939), 242-258.

\bibitem{Eth2} I.M.~Etherington, {\it Commutative train algebras of
rank 2 and 3} J.~London Math. Soc. \textbf{15}, (1949), 136-149.

\bibitem{Jacobson} N.~Jacobson, {\it Structure and Representations of Jordan Algebras}, Amer. Math. Soc., Providence, Rhode Island, 1968.

\bibitem{LZ} A.~Labra, A.~Suazo, {\it On plenary train algebras of rank
4} Comm.Algebra \textbf{35}, (2007), 2744-2752.

\bibitem{LR1}  J.~L\'opez-S\'anchez, E.~Rodr\'{\i}guez Santa Mar\'{\i}a, {\it Multibaric algebras} Non-Associative Algebra and its Applications, 235-240 Kluwer Academic Publishers  (1994).

\bibitem{LR2}  J.~L\'opez-S\'anchez, E.~Rodr\'{\i}guez Santa Mar\'{\i}a, {\it On
train algebras of rank 4} Comm. Algebra \textbf{24} (14), (1996), 4439-4445.

\bibitem{MS} C.~Mallol, A.~Suazo, {\it Une classe d'alg\`ebres pond\`er\'ees de degr\'ee quatre,} Comm. Algebra \textbf{28} (2000) 4, 2191-2199.

\bibitem{MVB} C.~Mallol, R.~Varro, R.~Benavides, {\it Gam\'etisation d'alg\`ebres pond\`er\'ees,} J.~Algebra \textbf{261} (2003), 1-18.

\bibitem{MV} C.~Mallol, R.~Varro, {A propos des Alg\`ebres V\'erifiant $x^{[3]} = w(x)^3x,$} Linear Algebra Appl. \textbf{225}, (1995), 187 - 194.

\bibitem{McC} K.~McCrimmon, {\it Generically algebraic algebras}, Trans. Amer. Math. Soc. \textbf{127} (1967), 527-551.

\bibitem{O} M.~Ouattara, {\it Sur les T-alg\`ebres de Jordan} Linear Algebra Appl. \textbf{144}, (1991), 11-21.

\bibitem{W} S.~Walcher, {\it Algebras which satisfy a train equation for the first three plenary powers} Arch.~Math. \textbf{56}, (1991), 547 - 551.

\bibitem{ZSSS}    K.A.~Zhevlakov, A.M.~Slinko, I.P.~Shestakov, A.I.~Shirshov, Rings that are nearly associative, Academic Press, New York, 1982.


\end{thebibliography}
\end{document}